\documentclass[11pt,oneside,english]{amsart}
\usepackage{amssymb}
\usepackage[colorlinks]{hyperref}
\usepackage{graphicx}
\usepackage[british]{babel}
\usepackage{microtype} 
\usepackage{tikz}
\usepackage{pgfplots}
\pgfplotsset{compat=1.18}

\makeatletter 
\theoremstyle{plain}
 \newtheorem{thm}{Theorem}[section]

 \numberwithin{equation}{section} 
 \numberwithin{figure}{section} 
 \theoremstyle{definition}
 
 \newtheorem{rem}[thm]{Remark}

\newcommand{\calD}{{{\mathcal D}}}
\newcommand{\bH}{{{\bf H}}}
\newcommand{\C}{{{\mathbb C}}}
\newcommand{\R}{{{\mathbb R}}}

\makeatother

\begin{document}

\title[Extremality of hyperbolic spiral and stretch maps]
{Extremality of hyperbolic spiral and stretch maps}

\author{Ioannis D. Platis}

\address{Department of Mathematics,
University of Patras, 
26504 Rion, Achaia,
 Greece.}
\email{idplatis@upatras.gr}


\begin{abstract}
We prove extremality results concerning specific quasiconformal mappings in the hyperbolic plane, namely, hyperbolic spiral and stretch maps. 
\end{abstract}

\maketitle

\section{Introduction}
Quasiconformal mappings serve as a fundamental tool in complex analysis, geometric function theory, and Teichm\"uller theory, providing a robust geometric framework to study the deformation of Riemann surfaces \cite{Ahlfors, Astala, Lehto}. A central theme in this field is Teichm\"uller's extremality problem, which seeks to identify a mapping that minimises the maximal distortion, $K_f$, within a given homotopy class and under prescribed boundary conditions. While the existence and theoretical characterisation of these extremal maps are well-established for general Riemann surfaces via the theory of quadratic differentials \cite{Gardiner, Strebel, Teichmuller}, deriving exact, closed-form formulae for such mappings on specific non-Euclidean domains remains a remarkably rare and challenging endeavour.

Alongside bounds on quasiconformal distortion, there is significant geometric interest in mappings that preserve hyperbolic area, which in two dimensions correspond to symplectic quasiconformal mappings \cite{Ahlfors}. Imposing this strict volume-preserving condition often creates a rigid geometric constraint that fundamentally conflicts with the minimisation of quasiconformal dilatation. Understanding how local distortion balances with symplectic constraints yields deep mathematical insights, particularly when mapped across distinct geometric moduli.

In this paper, we investigate the extremal properties of two highly structured classes of mappings on the hyperbolic plane $\bH^1_\C$: hyperbolic spiral maps and hyperbolic stretch maps. Our primary objective is to establish their exact extremality properties under specific boundary conditions and weighted mean distortion functionals.

We focus on these maps for three main reasons:
\begin{enumerate}
    \item {\it Sub-Riemannian geometry.} These planar maps are not just 2D models \cite{BubaniThesis}. They extend to contact quasiconformal mappings on solvable Lie groups, connecting standard complex function theory with sub-Riemannian geometry \cite{KoranyiReimann}.
    \item {\it Hyperbolic Riemann surfaces.} The spiral map $S_k$ gives a clear model for twisting the boundary along simple closed geodesics. Using the Collar Theorem \cite{Buser, Keen, Wolpert}, we can view the space around these geodesics as a hyperbolic circular annulus, $A_{1,R}$. Finding the extremality of $S_k$ tells us the minimum distortion needed for this twist.
    \item {\it Higher Teichm\"uller theory and complex hyperbolic geometry.} In higher Teichm\"uller theory and the geometry of character varieties (such as surface group representations into ${\rm SU}(2,1)$ or ${\rm SL}(3,\mathbb{R})$), generalised Fenchel--Nielsen coordinates and quasifuchsian deformations rely on decomposing geometric actions into radial stretches and torsional twist-shears along boundary curves or bisectors. Evaluating the exact extremal distortion of these elementary 2D spiral and stretch mappings provides fundamental building blocks and local distortion bounds for boundary extensions of complex hyperbolic structures on the complex hyperbolic plane $\bH^2_\C$.
\end{enumerate}

Even though a spiral map shears and a stretch map dilates, we prove their extremality using the same geometric approach. By using the modulus inequality on specific curve families, we avoid needing Teichm\"uller's theorem entirely. We show that the maximum distortion bounds for both maps come from the same basic proof.

The paper is organised as follows. In Section \ref{sec-prel}, we set up the geometric framework, introducing hyperbolic polar coordinates, explicit moduli of hyperbolic circular annuli, and the analytic properties of symplectic quasiconformal mappings. In Section \ref{sec-extremal}, we state and prove our main extremality results through isothermal coordinates: Section \ref{sec-spiral} treats the hyperbolic spiral map (Theorem \ref{thm-spiral}); Section \ref{sec-stretch} analyses both almost symplectic and non-symplectic stretch maps (Theorem \ref{thm-stretch}); Section \ref{sec-mean} proves extremality under mean distortion functionals (Theorems \ref{thm-mean1} and \ref{thm-mean2}); and Section \ref{sec-expressions} provides explicit closed-form expressions for all these maps directly in terms of the complex coordinate on the upper half-plane. Section \ref{sec-alternative} presents our unified alternative proofs via hyperbolic moduli, and Section \ref{sec-conclusion} provides concluding remarks.

\section{Preliminaries}\label{sec-prel}
In this section, we set up the geometric framework for our study. We outline hyperbolic polar coordinates in Section \ref{sec-hypcoord} and also the calculation of moduli for elementary domains, using standard methods (Section \ref{sec-modulus}). We also briefly review symplectic quasiconformal mappings in Section \ref{sec-symplectic}. Finally, in Section \ref{sec-calc} we present a formula for the Beltrami coefficient of a qusiconformal mapping in terms of polar coordinates and we introduce almost symplectic mappings.

\subsection{Hyperbolic polar coordinates}\label{sec-hypcoord} 
Consider the hyperbolic plane $\bH^1_\C$ with coordinates $z=\lambda+it$, $\lambda>0$, $t\in\R$. As detailed in \cite{Platis2025}, the hyperbolic polar coordinates map $\Phi:\calD=[0,\infty)\times[0,2\pi)\to\bH^1_\C$ is given by
\[
\Phi(r,\theta)=\left(\frac{1}{\cosh r-\cos\theta\sinh r},\;\frac{\sin\theta\sinh r}{\cosh r-\cos\theta\sinh r}\right).
\]
The inverse mapping uniquely defines $r$ and $\theta$ for any point in $\bH^1_\C$, where
\[
r=\operatorname{arccosh}\left(\frac{\lambda^2+t^2+1}{2\lambda}\right).
\]
The complete piecewise definition of $\theta$ is uniquely determined in $[0,2\pi)$ by the Cartesian coordinates \cite{Platis2025}. 
The hyperbolic metric tensor $g$ and symplectic form $\omega$ in these coordinates take the simple geometric forms
\[
g=dr^2+\sinh^2 r\,d\theta^2,\quad
\omega=\sinh r\,dr\wedge d\theta.
\]

\subsection{Modulus of the hyperbolic circular annulus}\label{sec-modulus}
Let $\Gamma$ be a family of curves in a domain $D \subset \bH^1_\C$. The modulus of $\Gamma$ with respect to the hyperbolic area element $d\mathcal{A}_h$ is defined as
\[
\operatorname{Mod}(\Gamma) = \inf_{\rho} \iint_D \rho^2 \, d\mathcal{A}_h,
\]
where the infimum is taken over all admissible metrics $\rho$; that is, all non-negative Borel measurable functions $\rho$ on $D$ such that \[\int_\gamma \rho \, ds_h \ge 1,\] for every locally rectifiable curve $\gamma \in \Gamma$.

Following the exact analytical framework established in~ \cite{Platis2025}, we consider the domain $A_{1,R}$ enclosed by two hyperbolic circles $r=1$ and $r=R>1$ and  we are able to derive exact expressions for the modulus of connecting and separating curve families within a hyperbolic circular annulus.

First, for the family of curves $\Gamma$ connecting the boundary components of the annulus $A_{1,R}$, the exact modulus is precisely given by
\[
\operatorname{Mod}(\Gamma)=\frac{2\pi}{\ln\left(\frac{\tanh(R/2)}{\tanh (1/2)}\right)};
\]
the corresponding extremal metric density for the connecting family $\Gamma$ is 
\[
\rho_{\Gamma}(r,\theta)=\frac{1}{\sinh r\ln\left(\frac{\tanh(R/2)}{\tanh(1/2)}\right)}.
\]
Secondly, for the subfamily of closed curves $\Gamma'$ separating the boundary components, that is, winding around the annulus, the modulus is
\[
\operatorname{Mod}(\Gamma')=\frac{1}{2\pi}\ln\left(\frac{\tanh(R/2)}{\tanh (1/2)}\right),
\]
and the corresponding extremal metric density for the separating family $\Gamma'$ is
\[
\rho_{\Gamma'}(r,\theta)=\frac{1}{2\pi\sinh r}.
\]
Notice that
\[
\operatorname{Mod}(\Gamma)\operatorname{Mod}(\Gamma')=1.
\]

\begin{figure}[htbp]
    \centering
    \begin{tabular}{cc}
        \begin{tikzpicture}
            \begin{axis}[
                width=3.2in, height=3.2in,
                view={35}{45},
                grid=major,
                colormap/viridis,
                xlabel={$\Re(w)$}, ylabel={$\Im(w)$}, zlabel={$\rho_{\Gamma'}(r)$},
                zlabel style={rotate=-90},
                title={Extremal Density $\rho_{\Gamma'}$ (Separating)},
                zmin=0, zmax=0.2,
                tick label style={font=\footnotesize},
                label style={font=\small}
            ]
            \addplot3 [
                surf,
                domain=1:3,
                domain y=0:360,
                samples=40,
                samples y=60,
                z buffer=sort
            ]
            ({tanh(x/2)*cos(y)}, {tanh(x/2)*sin(y)}, {1/(2*pi*sinh(x))});
            \end{axis}
        \end{tikzpicture}
        &
        \begin{tikzpicture}
            \begin{axis}[
                width=3.2in, height=3.2in,
                grid=major,
                xlabel={Hyperbolic Radius $r$},
                ylabel={Metric Density $\rho(r)$},
                title={Density Profiles for $R=3$},
                legend pos=north east,
                xmin=1, xmax=3,
                ymin=0, ymax=1.5,
                tick label style={font=\footnotesize},
                label style={font=\small}
            ]
            
            \addplot [
                blue!80!black, thick, smooth,
                domain=1:3,
                samples=100
            ] {1/(2*pi*sinh(x))};
            \addlegendentry{$\rho_{\Gamma'}$ (Separating)}
            
            \addplot [
                red!80!black, thick, smooth,
                domain=1:3,
                samples=100
            ] {1/(ln(tanh(3/2)/tanh(1/2)) * sinh(x))};
            \addlegendentry{$\rho_{\Gamma}$ (Connecting)}
            \end{axis}
        \end{tikzpicture}
        \\
        (A) 3D Surface over the Poincar\'e Annulus & (B) 2D Radial Profiles
    \end{tabular}
    \caption{Extremal metric densities for the canonical curve families in the hyperbolic annulus $A_{1,3}$. The 3D surface (A) maps the separating density $\rho_{\Gamma'}$ over the spatial coordinates of the unit disk model $w = \tanh(r/2)e^{i\theta}$. The cross-section (B) highlights the shared $1/\sinh r$ decay profile scaled by their respective moduli constants.}
    \label{fig:metric_densities}
\end{figure}
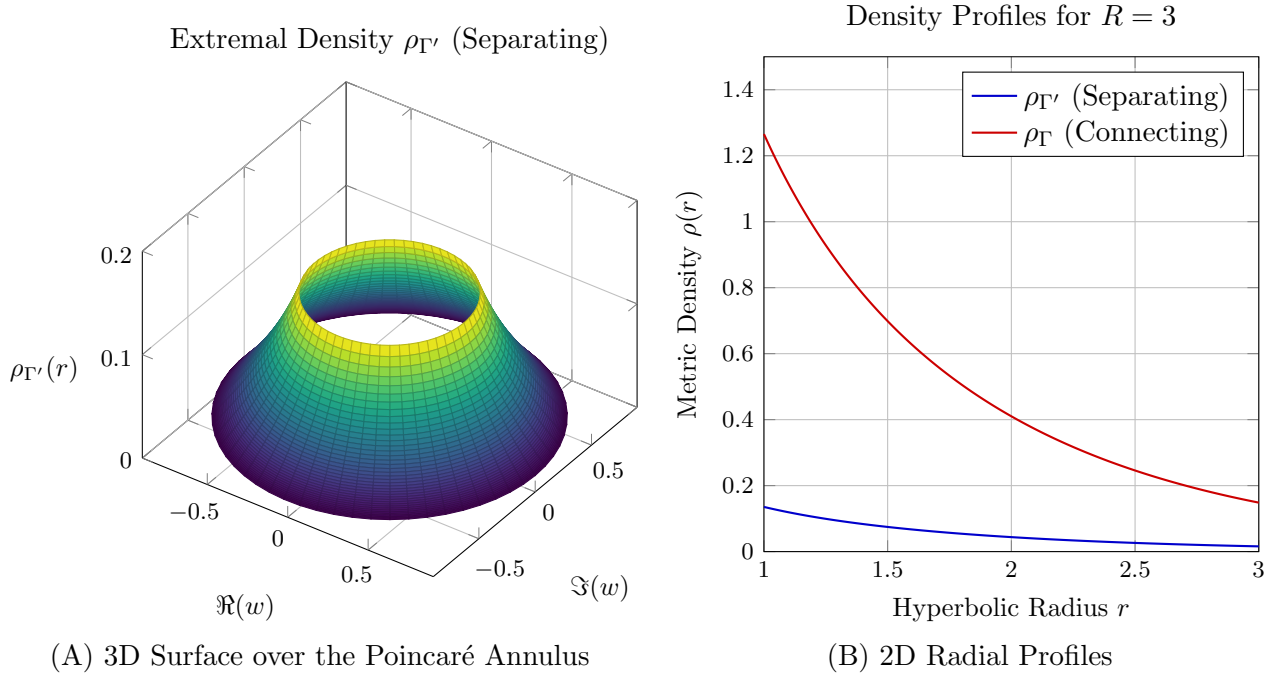
\subsection{Symplectic quasiconformal mappings}\label{sec-symplectic}
For background on symplectic quasiconformal mappings, we refer to \cite{Ahlfors} and also to standard texts on volume preservation and quasiconformality \cite{Lehto}. Let $f:\bH^1_\C\to \bH^1_\C$ be an orientation-preserving homeomorphism. The map $f$ is quasiconformal if its Beltrami coefficient satisfies
\[
|\mu_f(z)|=\left|\frac{f_{\overline z} }{f_z}\right|\le k<1,
\]
almost everywhere (a.e.) in $\bH^1_\C$. The dilatation (or distortion) is given by
\[
K_f(z)=\frac{1+|\mu_f(z)|}{1-|\mu_f(z)|} \le \frac{1+k}{1-k},
\]
a.e. in $\bH^1_\C$. The {\it maximal dilation} is
$$
K(f)={\rm ess sup}_{z\in\bH_\C^1}K_f(z).
$$
A diffeomorphism $f$ as above is \textit{symplectic} if it preserves the symplectic form $\omega$, i.e., $f^*\omega=\omega$. This geometric condition strictly constrains the Jacobian $J_f$, specifically \[
J_f(z)=\frac{\Re^2(f(z))}{\Re^2(z)},\quad z\in\bH^1_\C.\]
 A mapping that is both quasiconformal and symplectic represents a direct compromise between analytic distortion and geometric volume preservation \cite{Ahlfors}.

\subsection{Calculations}\label{sec-calc}
Let $\calD=[0,\infty)\times[0,2\pi)$ and suppose that $\tilde f:\calD\to \calD$, given by
\[
\tilde f(r,\theta)=(R(r,\theta),\,\Theta(r,\theta)),
\]
is a diffeomorphism. Then there exists a diffeomorphism $f:\bH^1_\C\to\bH^1_\C$ given by
\[
f(z)=f(\lambda,t)=(u(\lambda,t),\,v(\lambda,t))=u(z)+iv(z),
\]
such that $f\circ\Phi=\Phi\circ\tilde f$.
From the chain rule, we obtain
\begin{align*}
f_zz_r+f_{\overline{z}}\overline{z}_r &= z_RR_r+z_\Theta\Theta_r,\\
f_zz_\theta+f_{\overline{z}}\overline{z}_\theta &= z_RR_\theta+z_\Theta\Theta_\theta.
\end{align*}
We now note that
\[
z_r=-\frac{i}{\sinh r}z_\theta.
\]
Since the determinant of the system is $-2i\lambda^2\sinh r$, we obtain
\begin{align*}
    f_z &= \frac{\left|\begin{matrix}
    z_RR_r+z_\Theta\Theta_r&\overline{z}_r\\
    z_RR_\theta+z_\Theta\Theta_\theta&\overline{z}_\theta\end{matrix}\right|}{J_f}\\
    &= \frac{z_R\overline{z}_r\left|\begin{matrix}
    R_r-i\sinh R\,\Theta_r& 1
    \\
    R_\theta-i\sinh R\,\Theta_\theta&i\sinh r\end{matrix}\right|}{J_f},
\end{align*}
and similarly
\begin{align*}
f_{\overline{z}} &= \frac{z_Rz_r\left|\begin{matrix}
    1&R_r-i\sinh R\,\Theta_r
    \\
    -i\sinh r&R_\theta-i\sinh R\,\Theta_\theta\end{matrix}\right|}{J_f}.
\end{align*}
Therefore, we obtain the following expression for the Beltrami coefficient:
\begin{equation}\label{belt-abs}
    |\mu_f|=\left|\frac{(\sinh r\sinh R\,\Theta_r+R_\theta)+i(\sinh r\,R_r-\sinh R\,\Theta_\theta)}{(\sinh r\sinh R\,\Theta_r-R_\theta)+i(\sinh R\,\Theta_\theta+\sinh r\,R_r)}\right|.
\end{equation}

\begin{rem}
A diffeomorphism $f:\bH^1_\C\to\bH^1_\C$ is called \textit{almost symplectic} if there exists a positive constant $\sigma\in\R$ such that $f^*\omega=\sigma\omega$. In terms of the Euclidean Jacobian, this implies:
\[
J_f(z)=\sigma \frac{\Re^2(f(z))}{\Re^2(z)},\quad z\in\bH^1_\C.
\]
One might ask whether such a map can be normalised into a strictly symplectic map by a constant scaling. Let $F=\sigma^{a}f$ for some exponent $a$. The Euclidean Jacobian of $F$ scales quadratically:
\[
J_F(z)=\sigma^{2a}J_f(z)=\sigma^{2a+1}\frac{\Re^2(f(z))}{\Re^2(z)}=\frac{\Re^2(\sigma^{(2a+1)/2}f(z))}{\Re^2(z)}.
\]
However, for $F$ to be strictly symplectic, its Jacobian would need to satisfy the standard condition relative to its own real part: \[J_F(z) = \frac{\Re^2(F(z))}{\Re^2(z)} = \sigma^{2a}\frac{\Re^2(f(z))}{\Re^2(z)}.\]
 Equating the two conditions yields $\sigma^{2a+1} = \sigma^{2a}$, which is impossible unless $\sigma=1$. Thus, an almost symplectic map cannot be trivially rescaled into a symplectic one, highlighting that almost symplectic mappings form a genuinely distinct class under hyperbolic deformations.
\end{rem}

\section{Extremal mappings in the hyperbolic annulus}\label{sec-extremal}
In this section, we present our main extremality results for specific classes of mappings within the hyperbolic circular annulus. We begin in Section \ref{sec-extTeich} by recalling the Gr\"otzsch--Teichm\"uller extremality theorem, which provides the foundation for our proofs via isothermal coordinates. Section \ref{sec-spiral} treats the hyperbolic spiral map, establishing its constant distortion and extremality, while Section \ref{sec-stretch} provides a comparative analysis of almost symplectic and non-symplectic stretch maps, proving the extremality of the latter. In Section \ref{sec-mean}, we extend these results to show that both the spiral and non-symplectic stretch maps minimise specific weighted mean distortion functionals. Finally, Section \ref{sec-expressions} provides explicit closed-form formulae for all discussed maps directly in terms of complex coordinates.

\subsection{Extremal mappings and Teichm\"uller's theorem}\label{sec-extTeich}
Let $D,D'$ be domains in the complex plane. A quasiconformal mapping $f_0: D \to D'$ is \textit{extremal} in its homotopy class if its maximal dilatation satisfies $K(f_0) \le K(f)$ for all quasiconformal mappings $f$ that are homotopic to $f_0$ and share the same boundary values. By a \textit{Euclidean cylinder}, we refer to a Riemann surface formed by taking the quotient of a Euclidean vertical strip by a discrete group of translations, taking the form $C = (\R/2\pi\mathbb{Z}) \times [a,b]$ equipped with the standard flat metric $ds^2 = dx^2 + dy^2$.

Before examining specific mappings, we state Teichm\"uller's extremality theorem, which is the main tool for our proofs. By employing isothermal coordinates, we bridge our hyperbolic domains to Euclidean rectangles and cylinders, where this classical result applies directly. The following can be found in \cite[Chapter II]{Ahlfors} or \cite[Chapter I, Section 4.3, Theorem 4.1]{Lehto}.

\begin{thm}[Gr\"otzsch--Teichm\"uller Extremality Theorem]\label{thm:teich}
Let $F$ be an affine mapping (such as a linear stretch or a shear) between two Euclidean rectangles or cylinders. Then $F$ is the unique extremal mapping among all quasiconformal maps $f$ that lie in the same homotopy class and share the same boundary values. Specifically, the maximal dilatation satisfies $ K(f) \ge K(F)$, with equality holding if and only if $f = F$.
\end{thm}

Here, we investigate the extremality properties of specific classes of mappings within the hyperbolic annulus. We distinguish two primary cases: the hyperbolic spiral map and the hyperbolic stretch maps, examining their distortion and symplectic characteristics.

\subsection{The Hyperbolic Spiral Map}\label{sec-spiral}
We begin with the spiral map, which exhibits constant distortion and symplectic properties. The hyperbolic spiral map $S_k$, $k\in\R$ is defined in hyperbolic polar coordinates by
\begin{equation}\label{spiralmap}
\tilde S_k(r,\theta)=\left (r,\,\theta+k\ln(\tanh (r/2))\right).
\end{equation}

\begin{thm}\label{thm-spiral}
 The spiral map $S_k$:
 \begin{itemize}
     \item Maps the annulus $A_{1,R}$ to itself.
     \item Is symplectic.
     \item Is quasiconformal with constant distortion
     \[
     K_{S_k}=\frac{\sqrt{k^2+4}+|k|}{\sqrt{k^2+4}-|k|}.
     \]
     \item Is extremal for the family of quasiconformal mappings that map the annulus $A_{1,R}$ to itself and match its boundary values.
 \end{itemize}
\end{thm}

\begin{proof}
    The first assertion holds by definition, since $A_{1,R}=\{(r,\theta)\;|\; 1\le r\le R\}$ is preserved by a pure angular shift. 
    
    For the second assertion, the symplectic condition in hyperbolic polar coordinates is 
    \[
    \sinh R(R_r\Theta_\theta-R_\theta\Theta_r)=\sinh r.
    \] 
    In our case, $R=r$, $\Theta_\theta=1$, and $R_\theta=0$, which trivially satisfies the identity, confirming the map is symplectic.

    For the analytic distortion, straightforward differentiation yields $\Theta_r=\frac{k}{\sinh r}$. Calculating the Beltrami coefficient explicitly provides:
    \[
    |\mu_{S_k}|=\frac{\sinh r|\Theta_r|}{|\sinh r\Theta_r-2i|} = \frac{|k|}{\sqrt{k^2+4}}.
    \]
    This immediately verifies the third assertion regarding the constant maximal dilatation $K_{S_k}$.

    To establish extremality, we construct a conformal coordinate transformation 
    \[
    y(r) = \int \frac{dr}{\sinh r} = \ln(\tanh(r/2)),\quad x(\theta) = \theta.
    \]  
    Under this transformation, the hyperbolic metric takes the conformal form $ds_h^2 = \sinh^2 r(dx^2 + dy^2)$. Consequently, the hyperbolic annulus $A_{1,R}$ is conformally equivalent to a Euclidean cylinder 
    \[
    C = (\R/2\pi\mathbb{Z}) \times [y(1), y(R)].
    \] 
    In these conformal coordinates, the spiral map $S_k$ translates directly to the Euclidean mapping $F_k(x,y) = (x + ky, y)$, which represents a linear affine shear on the cylinder. By Theorem \ref{thm:teich}, such an affine shear is the unique extremal map among all quasiconformal maps of the cylinder to itself within the same homotopy class and matching the same boundary values. Because conformal coordinate transformations strictly preserve the maximal quasiconformal distortion, $S_k$ inherits this extremality property for the hyperbolic annulus $A_{1,R}$.
\end{proof}

Figure \ref{fig:spiral} illustrates the geometric twisting deformation of the standard hyperbolic polar grid under the spiral map on the circular annulus.

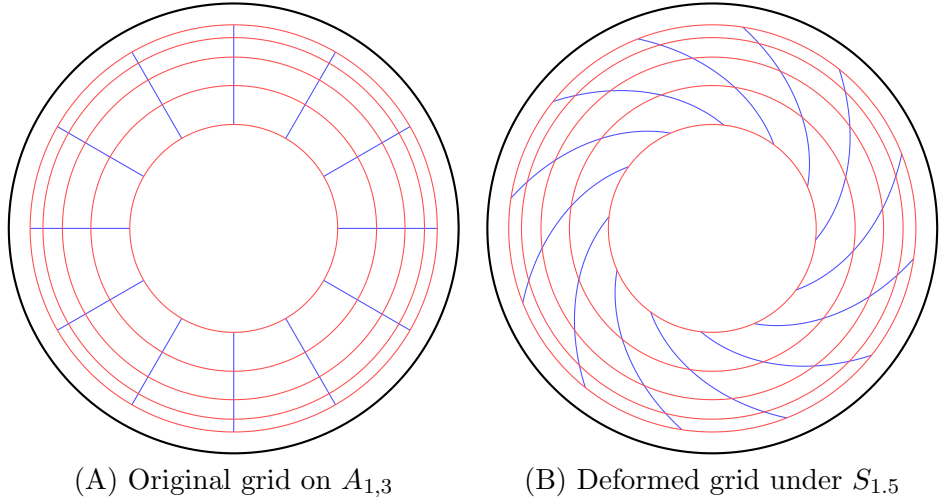
\begin{figure}[htbp]
    \centering
    \begin{tabular}{cc}
        \begin{tikzpicture}
        \begin{axis}[
            width=3.2in, height=3.2in,
            axis equal, hide axis,
            ymin=-1.1, ymax=1.1, xmin=-1.1, xmax=1.1,
        ]
        \addplot [black, thick, smooth, domain=0:360, samples=100] ({cos(\x)}, {sin(\x)});
        \pgfplotsinvokeforeach{0,30,...,330}{
            \addplot [blue!70, smooth, domain=1:3, samples=100] ({tanh(\x/2)*cos(#1)}, {tanh(\x/2)*sin(#1)});
        }
        \pgfplotsinvokeforeach{1, 1.5, 2, 2.5, 3}{
            \addplot [red!70, smooth, domain=0:360, samples=100] ({tanh(#1/2)*cos(\x)}, {tanh(#1/2)*sin(\x)});
        }
        \end{axis}
        \end{tikzpicture}
        &
        \begin{tikzpicture}
        \begin{axis}[
            width=3.2in, height=3.2in,
            axis equal, hide axis,
            ymin=-1.1, ymax=1.1, xmin=-1.1, xmax=1.1,
        ]
        \addplot [black, thick, smooth, domain=0:360, samples=100] ({cos(\x)}, {sin(\x)});
        \pgfplotsinvokeforeach{0,30,...,330}{
            \addplot [blue!70, smooth, domain=1:3, samples=100] ({tanh(\x/2)*cos(#1 + deg(1.5*ln(tanh(\x/2))))}, {tanh(\x/2)*sin(#1 + deg(1.5*ln(tanh(\x/2))))});
        }
        \pgfplotsinvokeforeach{1, 1.5, 2, 2.5, 3}{
            \addplot [red!70, smooth, domain=0:360, samples=100] ({tanh(#1/2)*cos(\x + deg(1.5*ln(tanh(#1/2))))}, {tanh(#1/2)*sin(\x + deg(1.5*ln(tanh(#1/2))))});
        }
        \end{axis}
        \end{tikzpicture}
        \\
        (A) Original grid on $A_{1,3}$ & (B) Deformed grid under $S_{1.5}$
    \end{tabular}
    \caption{Geometric deformation of the hyperbolic annulus under the spiral map $S_k$, visualised in the Poincar\'e disk model with $R=3$ and $k=1.5$.}
    \label{fig:spiral}
\end{figure}

\subsection{The Hyperbolic Stretch Maps}\label{sec-stretch}
Next, we examine stretch maps, contrasting the almost symplectic and non-symplectic variants, and establish an extremality result for the non-symplectic case.

The \textit{almost symplectic stretch map} $s'_k$, for $k \ge \frac{1}{\cosh 1}$, is defined by the equation
\begin{equation}\label{symp-str}
s'_k(r,\theta)=\left(\operatorname{arccosh}(k\cosh r),\,\theta\right).  
\end{equation}
Calculating its derivatives gives a non-constant Beltrami coefficient:
\[
|\mu_{s'_k}|=\left|\frac{1-k}{1+k}\right|\left|\frac{k\cosh^2 r+1}{k\cosh^2 r-1}\right|.
\]
In contrast, the \textit{non-symplectic stretch map} $s_k$ is defined by:
\begin{equation}
    s_k(r,\theta)=\left(2\operatorname{arctanh}\left((\tanh (r/2))^k\right),\,\theta\right).
\end{equation}

\begin{thm}\label{thm-stretch}
    The (non-symplectic) stretch map $s_k$ (assuming $k \ge 1$):
    \begin{itemize}
        \item Maps the annulus $A_{1,R}$ to the annulus $A_{r_1, r_2}$, where
        \[
        r_1 = 2\operatorname{arctanh}((\tanh(1/2))^k) \quad \text{and} \quad r_2 = 2\operatorname{arctanh}((\tanh(R/2))^k).
        \]
        \item Is not symplectic.
        \item Is quasiconformal with constant distortion $K_{s_k}=k$.
        \item Is extremal among all quasiconformal maps $f$ that map $A_{1, R}$ to $s_k (A_{1, R})$ and match the boundary values of $s_k$.
    \end{itemize}
\end{thm}
\begin{proof}
We begin by proving the first three properties through direct analytic computation. Let the mapping be denoted by $s_k(r,\theta) = (R(r,\theta), \Theta(r,\theta))$, where 
\[
R(r,\theta) = 2\operatorname{arctanh}\left((\tanh(r/2))^k\right) \quad \text{and} \quad \Theta(r,\theta) = \theta.
\]
\subsubsection*{1. Mapping of the annulus.} 
For $r > 0$, the function $r \mapsto \tanh(r/2)$ is strictly increasing, and for $k \ge 1$, $x \mapsto x^k$ is strictly increasing on $(0,1)$. Thus, the composition $r \mapsto R(r,\theta)$ is a strictly monotonically increasing function. The angular component $\Theta(r,\theta) = \theta$ trivially maps the interval $[0,2\pi)$ to itself. Evaluating the radial component at the boundary circles $r=1$ and $r=R$ yields exactly $r_1 = 2\operatorname{arctanh}((\tanh(1/2))^k)$ and $r_2 = 2\operatorname{arctanh}((\tanh(R/2))^k)$. By monotonicity and continuity, $s_k$ maps the annulus $A_{1,R}$ bijectively to $A_{r_1, r_2}$.

\subsubsection*{2. Symplectic condition.}
The symplectic condition in hyperbolic polar coordinates requires $\sinh(R) (R_r \Theta_\theta - R_\theta \Theta_r) = \sinh r$. 
Since $\Theta_\theta = 1$ and $R_\theta = \Theta_r = 0$, this condition reduces to $\sinh(R) R_r = \sinh r$. 
We compute the left-hand side: the partial derivative of $R$ with respect to $r$ is:
\[
R_r = \frac{2}{1 - (\tanh(r/2))^{2k}} \cdot k (\tanh(r/2))^{k-1} \cdot \frac{1}{2\cosh^2(r/2)} = \frac{k (\tanh(r/2))^{k-1}}{(1 - (\tanh(r/2))^{2k})\cosh^2(r/2)}.
\]
We also have:
\[
\sinh(R) = \sinh\left(2\operatorname{arctanh}\left((\tanh(r/2))^k\right)\right) = \frac{2(\tanh(r/2))^k}{1 - (\tanh(r/2))^{2k}}.
\]
Multiplying these gives:
\[
\sinh(R) R_r = \frac{2k (\tanh(r/2))^{2k-1}}{(1 - (\tanh(r/2))^{2k})^2\cosh^2(r/2)}.
\]
Using the identity $\sinh r = 2\tanh(r/2)\cosh^2(r/2)$, we observe that for $k > 1$, $\sinh(R) R_r \neq \sinh r$. Therefore, $s_k$ is not symplectic.

\subsubsection*{3. Distortion.}
We calculate the Beltrami coefficient $\mu_{s_k}$ using the coordinate expression \eqref{belt-abs}. With $\Theta_r = 0$, $R_\theta = 0$, and $\Theta_\theta = 1$, the expression simplifies to:
\[
|\mu_{s_k}| = \left|\frac{i(\sinh r\, R_r - \sinh R)}{i(\sinh R + \sinh r\, R_r)}\right| = \left|\frac{\sinh r\, R_r - \sinh R}{\sinh r\, R_r + \sinh R}\right|.
\]
We evaluate the product $\sinh r\, R_r$:
\[
\sinh r\, R_r = 2\tanh(r/2)\cosh^2(r/2) \cdot \frac{k (\tanh(r/2))^{k-1}}{(1 - (\tanh(r/2))^{2k})\cosh^2(r/2)} = \frac{2k (\tanh(r/2))^k}{1 - (\tanh(r/2))^{2k}}.
\]
Comparing this with our previous expression for $\sinh R$, we find exactly that:
\[
\sinh r\, R_r = k \sinh R.
\]
Substituting this relationship into the Beltrami coefficient formula yields:
\[
|\mu_{s_k}| = \left|\frac{k \sinh R - \sinh R}{k \sinh R + \sinh R}\right| = \frac{k-1}{k+1},
\]
since $k \ge 1$. The maximal dilatation is thus strictly constant:
\[
K_{s_k} = \frac{1 + |\mu_{s_k}|}{1 - |\mu_{s_k}|} = \frac{1 + \frac{k-1}{k+1}}{1 - \frac{k-1}{k+1}} = \frac{2k}{2} = k.
\]

\subsubsection*{4. Extremality.}
To prove the extremality of $s_k$, we employ the same conformal coordinate system $y(r) = \ln(\tanh(r/2))$ and $x(\theta) = \theta$. Under this transformation, the hyperbolic annulus $A_{1,R}$ is conformally mapped to the Euclidean rectangle $C = [0,2\pi] \times [y(1), y(R)]$. In these coordinates, the non-symplectic stretch map $s_k(r,\theta) = (R,\Theta)$ transforms to a map $F_k(x,y) = (X(x,y), Y(x,y))$. 

Using the definition of $s_k$, we observe that $R = 2\operatorname{arctanh}((\tanh(r/2))^k)$, leading to:
\[
Y(R) = \ln(\tanh(R/2)) = \ln\left(\tanh\left(\operatorname{arctanh}((\tanh(r/2))^k)\right)\right) = \ln((\tanh(r/2))^k).
\]
By the properties of logarithms, this simplifies directly to:
\[
Y(R) = k \ln(\tanh(r/2)) = k y(r).
\]
Additionally, the angular coordinate is preserved, so $X(\Theta) = \Theta = \theta = x$. Thus, $s_k$ corresponds exactly to the linear affine stretch $F_k(x,y) = (x, ky)$. By Theorem \ref{thm:teich}, such an affine stretch is the unique extremal map among all quasiconformal maps between these domains in the same homotopy class and matching the same boundary values. Because conformal coordinate transformations strictly preserve the maximal quasiconformal distortion, $s_k$ inherits this extremality property for the hyperbolic annulus.
\end{proof}

Figure \ref{fig:stretch_nonsymp} displays the uniform radial compression induced by the non-symplectic stretch map, while Figure \ref{fig:stretch_symp} illustrates the deformation under the almost symplectic map.

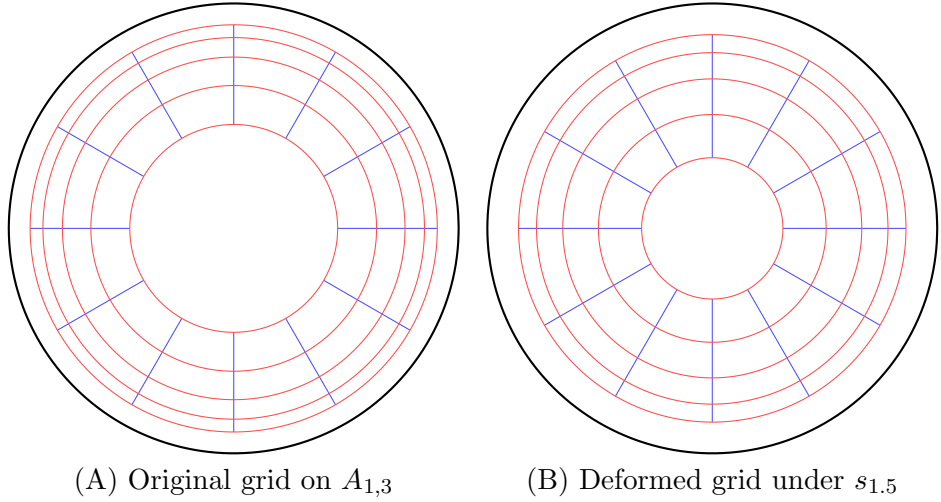
\begin{figure}[htbp]
    \centering
    \begin{tabular}{cc}
        \begin{tikzpicture}
        \begin{axis}[
            width=3.2in, height=3.2in,
            axis equal, hide axis,
            ymin=-1.1, ymax=1.1, xmin=-1.1, xmax=1.1,
        ]
        \addplot [black, thick, smooth, domain=0:360, samples=100] ({cos(\x)}, {sin(\x)});
        \pgfplotsinvokeforeach{0,30,...,330}{
            \addplot [blue!70, smooth, domain=1:3, samples=100] ({tanh(\x/2)*cos(#1)}, {tanh(\x/2)*sin(#1)});
        }
        \pgfplotsinvokeforeach{1, 1.5, 2, 2.5, 3}{
            \addplot [red!70, smooth, domain=0:360, samples=100] ({tanh(#1/2)*cos(\x)}, {tanh(#1/2)*sin(\x)});
        }
        \end{axis}
        \end{tikzpicture}
        &
        \begin{tikzpicture}
        \begin{axis}[
            width=3.2in, height=3.2in,
            axis equal, hide axis,
            ymin=-1.1, ymax=1.1, xmin=-1.1, xmax=1.1,
        ]
        \addplot [black, thick, smooth, domain=0:360, samples=100] ({cos(\x)}, {sin(\x)});
        \pgfplotsinvokeforeach{0,30,...,330}{
            \addplot [blue!70, smooth, domain=1:3, samples=100] ({pow(tanh(\x/2), 1.5)*cos(#1)}, {pow(tanh(\x/2), 1.5)*sin(#1)});
        }
        \pgfplotsinvokeforeach{1, 1.5, 2, 2.5, 3}{
            \addplot [red!70, smooth, domain=0:360, samples=100] ({pow(tanh(#1/2), 1.5)*cos(\x)}, {pow(tanh(#1/2), 1.5)*sin(\x)});
        }
        \end{axis}
        \end{tikzpicture}
        \\
        (A) Original grid on $A_{1,3}$ & (B) Deformed grid under $s_{1.5}$
    \end{tabular}
    \caption{Geometric deformation of the hyperbolic annulus under the non-symplectic stretch map $s_k$, visualised in the Poincar\'e disk model with $R=3$ and $k=1.5$. Notice the uniform radial compression across the entire annulus, characteristic of a map that minimises maximal quasiconformal distortion without regard for area preservation.}
    \label{fig:stretch_nonsymp}
\end{figure}

\begin{figure}[htbp]
    \centering
    \begin{tabular}{cc}
        \begin{tikzpicture}
        \begin{axis}[
            width=3.2in, height=3.2in,
            axis equal, hide axis,
            ymin=-1.1, ymax=1.1, xmin=-1.1, xmax=1.1,
        ]
        \addplot [black, thick, smooth, domain=0:360, samples=100] ({cos(\x)}, {sin(\x)});
        \pgfplotsinvokeforeach{0,30,...,330}{
            \addplot [blue!70, smooth, domain=1:3, samples=100] ({tanh(\x/2)*cos(#1)}, {tanh(\x/2)*sin(#1)});
        }
        \pgfplotsinvokeforeach{1, 1.5, 2, 2.5, 3}{
            \addplot [red!70, smooth, domain=0:360, samples=100] ({tanh(#1/2)*cos(\x)}, {tanh(#1/2)*sin(\x)});
        }
        \end{axis}
        \end{tikzpicture}
        &
        \begin{tikzpicture}
        \begin{axis}[
            width=3.2in, height=3.2in,
            axis equal, hide axis,
            ymin=-1.1, ymax=1.1, xmin=-1.1, xmax=1.1,
        ]
        \addplot [black, thick, smooth, domain=0:360, samples=100] ({cos(\x)}, {sin(\x)});
        \pgfplotsinvokeforeach{0,30,...,330}{
            \addplot [blue!70, smooth, domain=1:3, samples=100] ({sqrt((0.5+2.5*pow(tanh(\x/2), 2))/(2.5+0.5*pow(tanh(\x/2), 2)))*cos(#1)}, {sqrt((0.5+2.5*pow(tanh(\x/2), 2))/(2.5+0.5*pow(tanh(\x/2), 2)))*sin(#1)});
        }
        \pgfplotsinvokeforeach{1, 1.5, 2, 2.5, 3}{
            \addplot [red!70, smooth, domain=0:360, samples=100] ({sqrt((0.5+2.5*pow(tanh(#1/2), 2))/(2.5+0.5*pow(tanh(#1/2), 2)))*cos(\x)}, {sqrt((0.5+2.5*pow(tanh(#1/2), 2))/(2.5+0.5*pow(tanh(#1/2), 2)))*sin(\x)});
        }
        \end{axis}
        \end{tikzpicture}
        \\
        (A) Original grid on $A_{1,3}$ & (B) Deformed grid under $s'_{1.5}$
    \end{tabular}
    \caption{Geometric deformation of the hyperbolic annulus under the almost symplectic stretch map $s'_k$, visualised in the Poincar\'e disk model with $R=3$ and $k=1.5$. In contrast to Figure \ref{fig:stretch_nonsymp}, observe the non-linear, non-uniform radial deformation required to satisfy the almost symplectic constraint, which forces a distinct spatial distribution of area.}
    \label{fig:stretch_symp}
\end{figure}
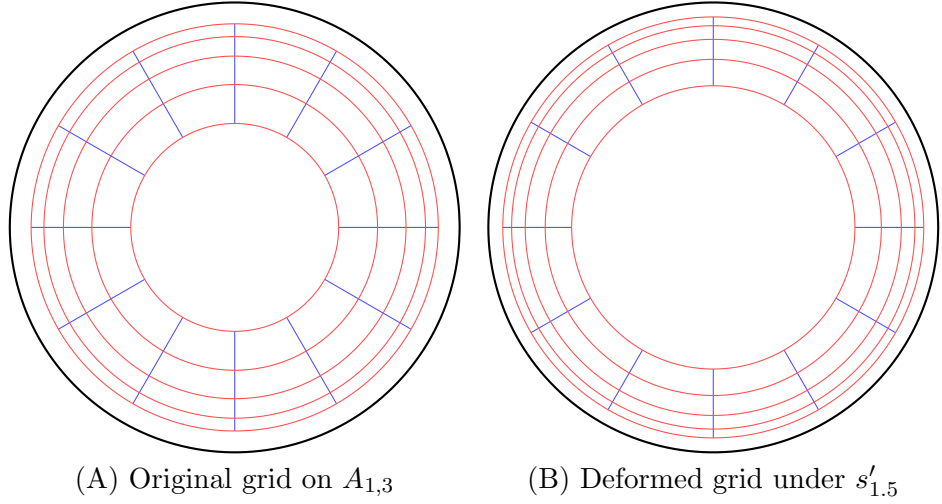

\begin{rem}
The almost symplectic stretch map $s'_k$ defined in \eqref{symp-str} translates into a strictly non-linear mapping in the conformal coordinates $(x,y)$. To see this explicitly, recall the conformal coordinate transformation $y(r) = \ln(\tanh(r/2))$ and $x(\theta) = \theta$. Under this change of variables, the new radial coordinate $R = \operatorname{arccosh}(k\cosh r)$ becomes:
\[
Y(y) = \ln(\tanh(R/2)) = \frac{1}{2}\ln\left(\frac{\cosh R - 1}{\cosh R + 1}\right) = \frac{1}{2}\ln\left(\frac{k\cosh r(y) - 1}{k\cosh r(y) + 1}\right).
\]
Differentiating $Y$ with respect to $y$ yields a non-constant derivative:
\[
Y'(y) = \frac{k\sinh^2 r(y)}{k^2\cosh^2 r(y) - 1}.
\]
Because $Y'(y)$ varies continuously with $y$, the corresponding function $F_k(x,y) = (x, Y(y))$ on the Euclidean cylinder is non-linear. Consequently, it possesses a non-constant Beltrami coefficient. By Theorem \ref{thm:teich}, it cannot be extremal in its homotopy class for matching boundary values. This highlights a fundamental tension in hyperbolic deformations: imposing a strict volume-preserving (symplectic) constraint on a radial stretch dictates a non-linear mapping profile in conformal coordinates, directly preventing the mapping from globally minimising maximal quasiconformal distortion.
\end{rem}
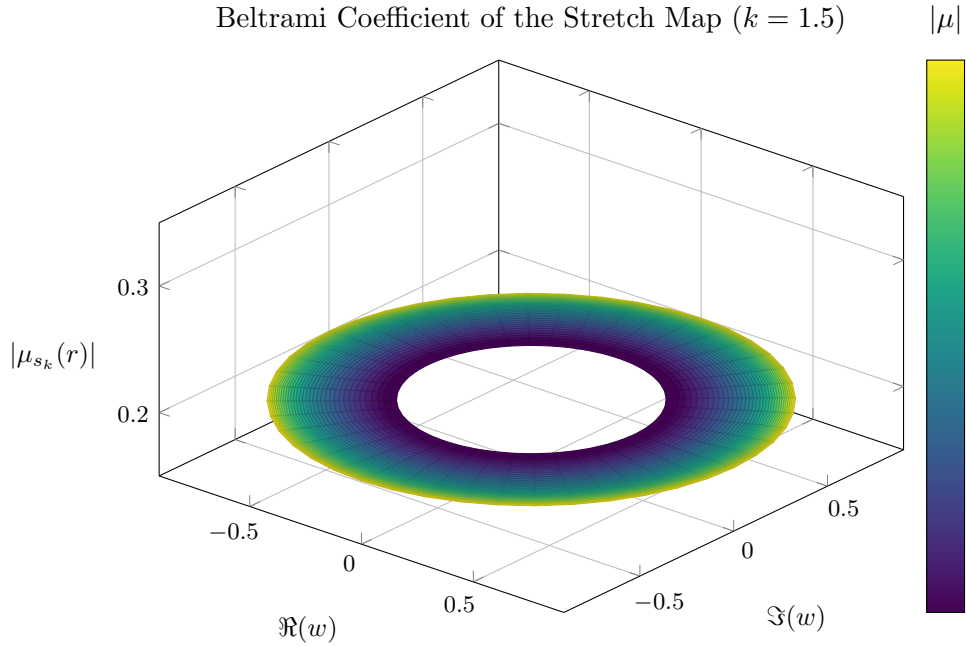
\begin{figure}[htbp]
    \centering
    \begin{tikzpicture}
        \begin{axis}[
            width=4.5in, height=3.5in,
            view={40}{40},
            grid=major,
            colormap/viridis,
            xlabel={$\Re(w)$}, ylabel={$\Im(w)$}, zlabel={$|\mu_{s_k}(r)|$},
            zlabel style={rotate=-90},
            title={Beltrami Coefficient of the Stretch Map ($k=1.5$)},
            zmin=0.15, zmax=0.35,
            tick label style={font=\footnotesize},
            label style={font=\small},
            colorbar,
            colorbar style={
                title={$|\mu|$},
                ytick={0.2, 0.25, 0.3},
                tick label style={font=\footnotesize}
            }
        ]
        
        \addplot3 [
            surf,
            domain=1:3,
            y domain=0:360,
            samples=35,
            samples y=45,
            z buffer=sort
        ]
        (
            { ((exp(x)-1)/(exp(x)+1)) * cos(y) }, 
            { ((exp(x)-1)/(exp(x)+1)) * sin(y) }, 
            { 0.2 + 0.00001*x }
        );
        \end{axis}
    \end{tikzpicture}
    \caption{The constant Beltrami coefficient $|\mu_{s_k}|$ of the standard stretch map $s_k$ plotted over the spatial coordinates of the annulus $A_{1,3}$. The strict flatness of this surface ($|\mu_{s_k}| \equiv 0.2$) visually supports the extremality of the mapping within its homotopy class.}
    \label{fig:beltrami_nonsymp}
\end{figure}
\begin{figure}[htbp]
    \centering
    \begin{tikzpicture}
        \begin{axis}[
            width=4.5in, height=3.5in,
            view={40}{40},
            grid=major,
            colormap/viridis,
            xlabel={$\Re(w)$}, ylabel={$\Im(w)$}, zlabel={$|\mu_{s'_k}(r)|$},
            zlabel style={rotate=-90},
            title={Beltrami Coefficient of the Almost Symplectic Stretch Map ($k=1.5$)},
            zmin=0.15, zmax=0.35,
            tick label style={font=\footnotesize},
            label style={font=\small},
            colorbar,
            colorbar style={
                title={$|\mu|$},
                ytick={0.2, 0.25, 0.3},
                tick label style={font=\footnotesize}
            }
        ]
        
        \addplot3 [
            surf,
            domain=1:3,
            domain y=0:360,
            samples=35,
            samples y=45,
            variable=\u,
            variable y=\v,
            z buffer=sort
        ]
        (
            {(exp(\u) - 1)/(exp(\u) + 1) * cos(\v)}, 
            {(exp(\u) - 1)/(exp(\u) + 1) * sin(\v)}, 
            {0.2 * (1.5 * ((exp(\u) + exp(-\u))/2)^2 + 1) / (1.5 * ((exp(\u) + exp(-\u))/2)^2 - 1)}
        );
        \end{axis}
    \end{tikzpicture}
    \caption{The non-constant Beltrami coefficient $|\mu_{s'_k}|$ of the almost symplectic stretch map $s'_k$ plotted over the spatial coordinates of the annulus $A_{1,3}$ in the unit disk model. The radial variation in distortion visually confirms that $s'_k$ cannot be extremal in its homotopy class.}
    \label{fig:beltrami_symp}
\end{figure}
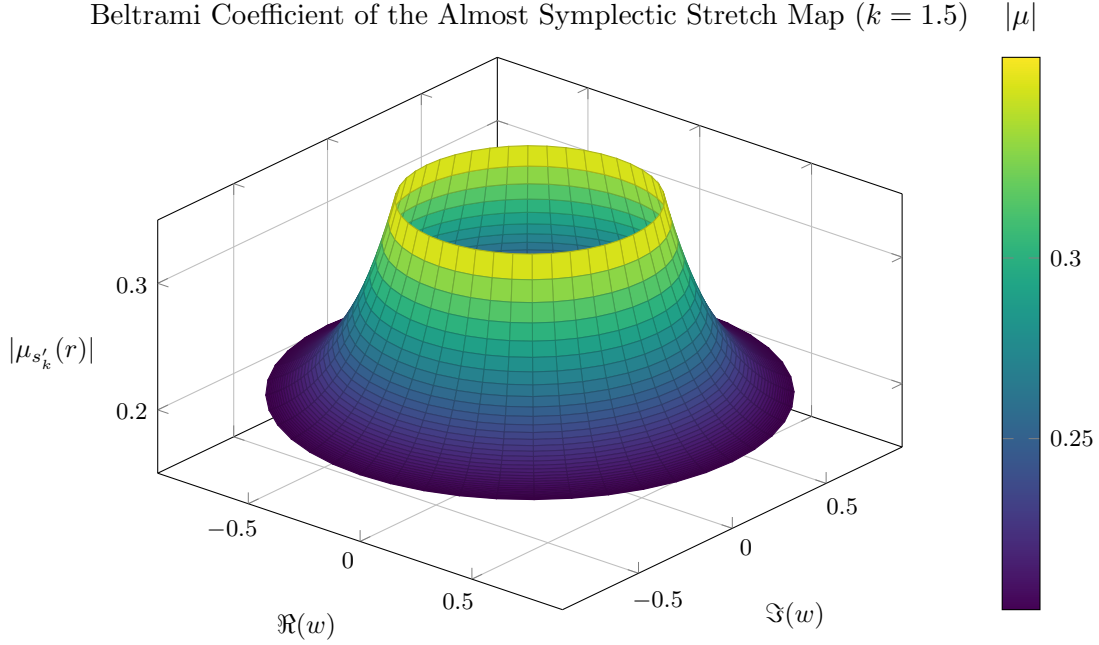

\subsection{Mean distortion extremality}\label{sec-mean}
In this section, we establish that both the hyperbolic spiral map and the non-symplectic stretch map minimise a specific mean distortion functional. Let $f$ be a quasiconformal mapping defined on the hyperbolic annulus $A_{1,R}$. We define the mean distortion functional as
\[
M(f) = \iint_{A_{1,R}} K_f(r,\theta) \frac{1}{\sinh^2 r} \, d\mathcal{A}_h.
\]
\begin{rem}
We stress at this point that geometrically, weighting the maximal dilatation $K_f$ by $\frac{1}{\sinh^2 r} \, d\mathcal{A}_h$ is the natural choice for the hyperbolic annulus. There are two simple reasons: first, it perfectly matches the square of the standard density function $\rho_{\Gamma'}$ for curves in the annulus. Second, and most importantly, hyperbolic area grows very rapidly as we move outward. Without this weight, the calculation would unfairly focus on the distortion happening at the outer edges. The factor $\frac{1}{\sinh^2 r}$ cancels out this rapid growth, ensuring we measure distortion equally across the entire shape.
\end{rem}

\begin{thm}\label{thm-mean1}
The hyperbolic spiral map $S_k$ and the non-symplectic stretch map $s_k$ are extremal for the mean distortion functional $M(f)$ among all quasiconformal maps that map $A_{1,R}$ to their respective images and match their boundary values.
\end{thm}

\begin{proof}
We can establish the extremality for both mappings simultaneously by exploiting their constant maximal dilatation and the fundamental modulus inequality, bypassing the need for Euclidean coordinate transformations entirely.

Let $f_0$ represent either the hyperbolic spiral map $S_k$ or the non-symplectic stretch map $s_k$. By Theorems \ref{thm-spiral} and \ref{thm-stretch}, both mappings possess a strictly constant maximal dilatation, which we denote by $K_0$. 

Let $\Gamma$ be the canonical family of curves in $A_{1,R}$ aligned with the major axis of the Beltrami ellipse of $f_0$. Specifically:
\begin{itemize}
    \item For $f_0 = s_k$, $\Gamma$ is the family of separating closed curves $\Gamma'$ winding around the annulus.
    \item For $f_0 = S_k$, $\Gamma$ is the family of logarithmic spirals $\Gamma_\alpha$ traversing the annulus at the constant angle dictated by the shear.
\end{itemize}
In both cases, there exists an extremal metric density $\rho_\Gamma(r,\theta)$ associated with $\Gamma$ such that its square is directly proportional to the geometric weight of our mean distortion functional:
\[
\rho_\Gamma^2(r,\theta) = \frac{C}{\sinh^2 r},
\]
for some geometric constant $C > 0$ (for instance, $C = \frac{1}{4\pi^2}$ for $\Gamma'$). Consequently, the mean distortion functional can be expressed entirely in terms of this extremal metric:
\[
M(f) = \frac{1}{C} \iint_{A_{1,R}} K_f(r,\theta) \rho_\Gamma^2(r,\theta) \, d\mathcal{A}_h.
\]
Because $f_0$ dilates the geometry uniformly by exactly $K_0$ along the curves of $\Gamma$, its action on the modulus is exact:
\[
\operatorname{Mod}(f_0(\Gamma)) = K_0 \operatorname{Mod}(\Gamma) = \iint_{A_{1,R}} K_{f_0} \rho_\Gamma^2 \, d\mathcal{A}_h = C \cdot M(f_0).
\]
Now, let $f$ be any competing quasiconformal map that is homotopic to $f_0$ and matches its boundary values. Since $f$ acts on the same boundary components, it maps the family $\Gamma$ to the topologically equivalent family $f(\Gamma) = f_0(\Gamma)$. Applying the fundamental weighted quasiconformal modulus inequality to $f$, we obtain:
\[
\operatorname{Mod}(f(\Gamma)) \le \iint_{A_{1,R}} K_f(r,\theta) \rho_\Gamma^2(r,\theta) \, d\mathcal{A}_h = C \cdot M(f).
\]
Equating the moduli since $f(\Gamma) = f_0(\Gamma)$, we find:
\[
C \cdot M(f_0) = \operatorname{Mod}(f_0(\Gamma)) = \operatorname{Mod}(f(\Gamma)) \le C \cdot M(f).
\]
Dividing by the positive constant $C$ immediately yields $M(f_0) \le M(f)$. This confirms that both the hyperbolic spiral map $S_k$ and the non-symplectic stretch map $s_k$ strictly achieve the absolute minimum for the mean distortion functional $M(f)$ within their respective homotopy classes.
\end{proof}

Before stating the extremality theorem for the almost symplectic stretch map $s'_k$, we clarify the explicit geometric significance of its specific weight function. By the chain rule, the derivative of the new conformal coordinate $Y$ with respect to the initial conformal coordinate $y$ evaluates to:
\[
Y'(y) = \frac{dY/dr}{dy/dr} = \frac{\frac{k\sinh r}{k^2\cosh^2 r - 1}}{\frac{1}{\sinh r}} = \frac{k\sinh^2 r}{k^2\cosh^2 r - 1}.
\]
This shows that the weight factor $\left( \frac{k\sinh^2 r}{k^2\cosh^2 r - 1} \right)^2$ is exactly $(Y'(y))^2$. We include this weight because the map has to compress the radial distance to preserve the area. Geometrically, this factor represents the squared Jacobian of the radial deformation in conformal coordinates. Because an almost symplectic stretch map must heavily compress the radial distance to compensate for area scaling, this specific weight ensures that the mean distortion functional heavily penalises regions undergoing the most severe radial compression, accurately capturing the geometric cost of the pseudo-volume-preserving constraint.

\begin{thm}\label{thm-mean2}
The almost symplectic stretch map $s'_k$ defined in \eqref{symp-str} uniquely minimises the weighted mean distortion functional 
\[
M_w(f) = \iint_{A_{1,R}} K_f(r,\theta) \left( \frac{k\sinh^2 r}{k^2\cosh^2 r - 1} \right)^2 \frac{1}{\sinh^2 r} \, d\mathcal{A}_h
\]
among all quasiconformal maps that map $A_{1,R}$ to $s'_k(A_{1,R})$ and match its boundary values.
\end{thm}

\begin{proof}
We transition to the conformal coordinate system $y(r) = \ln(\tanh(r/2))$ and $x(\theta) = \theta$, mapping $A_{1,R}$ to the Euclidean cylinder $C$. Recall the area relation $dx dy = \frac{1}{\sinh^2 r} \, d\mathcal{A}_h$. The map $s'_k$ corresponds to $F_k(x,y) = (x, Y(y))$, where $Y'(y) = \frac{k\sinh^2 r}{k^2\cosh^2 r - 1}$. 
For $k > 1$, we observe that $Y'(y) < 1$. Consequently, the principal stretches of $F_k$ are $1$ (in the $x$-direction) and $Y'(y)$ (in the $y$-direction). Its Jacobian is $J_{F_k} = Y'(y)$ and its maximal distortion is $K_{F_k} = 1/Y'(y)$. 

The geometric weight in conformal coordinates translates precisely to $w(y) = (Y'(y))^2$. Using the area relation, the weighted functional over the hyperbolic annulus maps directly to the following integral over the Euclidean cylinder $C$:
\[
M_w(f) = \iint_{C} K_f(x,y) (Y'(y))^2 \, dx dy.
\]
For the exact mapping $s'_k$, this evaluates to:
\[
M_w(s'_k) = \iint_{C} \frac{1}{Y'(y)} (Y'(y))^2 \, dx dy = \iint_{C} Y'(y) \, dx dy = \operatorname{Area}(F_k(C)).
\]
For any competing map $f(x,y) = (U(x,y), V(x,y))$ matching the boundary values, we use the pointwise quasiconformal inequality $K_f \ge U_x^2 / J_f$. Thus,
\[
M_w(f) \ge \iint_{C} \frac{U_x^2 (Y'(y))^2}{J_f} \, dx dy.
\]
Applying the Cauchy-Schwarz inequality to the functions $U_x Y'(y) / \sqrt{J_f}$ and $\sqrt{J_f}$:
\[
\left( \iint_{C} U_x Y'(y) \, dx dy \right)^2 \le \left(\iint_{C} \frac{U_x^2 (Y'(y))^2}{J_f} \, dx dy \right) \left( \iint_{C} J_f \, dx dy \right).
\]
Because $f$ matches the boundary of $F_k$ and wraps continuously around the cylinder, the horizontal derivative integrates to $\int_0^{2\pi} U_x \, dx = 2\pi$. Hence, the left-hand integral is exactly $2\pi \int Y'(y) \, dy = \operatorname{Area}(F_k(C))$. Furthermore, the total area integral of the Jacobian is simply the area of the target domain, $\iint_C J_f \, dx dy = \operatorname{Area}(F_k(C))$. 
Substituting these geometric quantities into the inequality yields:
\[
\operatorname{Area}(F_k(C))^2 \le M_w(f) \cdot \operatorname{Area}(F_k(C)) \implies M_w(f) \ge \operatorname{Area}(F_k(C)).
\]
Since $M_w(s'_k) = \operatorname{Area}(F_k(C))$, the almost symplectic stretch map strictly achieves the absolute theoretical minimum, completing the proof.
\end{proof}

\subsection{Expressions in complex coordinates}\label{sec-expressions}
To obtain explicit formulae for our mappings directly in terms of the complex coordinate $z = \lambda + it \in \bH^1_\C$, it is most convenient to pass through the unit disk $\Delta$ via the Cayley transform $w = \frac{z-1}{z+1}$. 

Recall from Section 2.1 that the polar coordinate map $\Phi(r,\theta)$ corresponds to the inverse Cayley transform $z = \frac{1+w}{1-w}$ evaluated at $w = \tanh(r/2)e^{i\theta}$. Thus, for any point $z \in \bH^1_\C$, we have the explicit relations
\[
\tanh(r/2) = |w| = \left|\frac{z-1}{z+1}\right|, \quad e^{i\theta} = \frac{w}{|w|} = \frac{z-1}{z+1}\left|\frac{z+1}{z-1}\right|.
\]
Furthermore, the hyperbolic term $\cosh r$ can be expressed directly in terms of $z$ as:
\[
\cosh r(z) = \frac{|z|^2+1}{2\Re(z)} = \frac{|z|^2+1}{z+\bar{z}}.
\]

Using these relations, the hyperbolic spiral map $S_k$ modifies the disk coordinate to $\tilde{w} = w e^{ik\ln|w|}$. Substituting this into the inverse Cayley transform yields the explicit closed-form formula in $\bH^1_\C$:
\[
S_k(z) = \frac{z+1 + (z-1)\left|\frac{z-1}{z+1}\right|^{ik}}{z+1 - (z-1)\left|\frac{z-1}{z+1}\right|^{ik}}.
\]

Similarly, the non-symplectic stretch map $s_k$ modifies the disk coordinate to $\tilde{w} = w |w|^{k-1}$. In the complex coordinate $z$, this gives:
\[
s_k(z) = \frac{z+1 + (z-1)\left|\frac{z-1}{z+1}\right|^{k-1}}{z+1 - (z-1)\left|\frac{z-1}{z+1}\right|^{k-1}}.
\]

Finally, for the almost symplectic stretch map $s'_k(r, \theta) = (\operatorname{arccosh}(k\cosh r), \theta)$, the new radial coordinate $R$ satisfies $\cosh R = k\cosh r$. In the unit disk, the modulus of the transformed coordinate $\tilde{w}$ must satisfy
\[
\frac{1+|\tilde{w}|^2}{1-|\tilde{w}|^2} = k\cosh r(z) \implies |\tilde{w}| = \sqrt{\frac{k\cosh r(z) - 1}{k\cosh r(z) + 1}}.
\]
Since the angular coordinate is preserved, the new disk coordinate is $\tilde{w} = w \frac{|\tilde{w}|}{|w|}$. Let us define the real scaling factor $A_k(z)$ as the ratio of the new and old moduli:
\[
A_k(z) = \frac{|\tilde{w}|}{|w|} = \left|\frac{z+1}{z-1}\right| \sqrt{\frac{k\cosh r(z) - 1}{k\cosh r(z) + 1}}.
\]
Applying the inverse Cayley transform to $\tilde{w} = w A_k(z) = \frac{z-1}{z+1} A_k(z)$ yields the formula for the almost symplectic stretch map:
\[
s'_k(z) = \frac{z+1 + (z-1)A_k(z)}{z+1 - (z-1)A_k(z)}.
\]

\section{Alternative Proofs via Hyperbolic Moduli}\label{sec-alternative}

In this section, we establish the extremality of the hyperbolic spiral and stretch maps directly within $\bH^1_\C$ using the modulus of curve families. This approach bypasses Teichm\"uller's theorem and conformal coordinate reductions, providing a purely intrinsic proof in hyperbolic geometry.

\subsection{Modulus Inequality in $\bH^1_\C$}
Recall that for any family of locally rectifiable curves $\Gamma$ in $A_{1,R}$ and any $K$-quasiconformal mapping $f: A_{1,R} \to A'$, the modulus satisfies the fundamental inequality
\begin{equation}\label{eq:mod-qc-ineq}
\operatorname{Mod}(f(\Gamma)) \le K(f)\cdot \operatorname{Mod}(\Gamma), \quad K(f)=\sup_{z \in A_{1,R}} K_f(z).
\end{equation}
More generally, if $K_f(z)$ is non-constant, then for any admissible density $\rho$ for $\Gamma$,
\begin{equation}\label{eq:mod-weighted-ineq}
\operatorname{Mod}(f(\Gamma)) \le \iint_{A_{1,R}} K_f(r,\theta) \rho^2(r,\theta) \, d\mathcal{A}_h.
\end{equation}

\subsection{Unified Proof of Extremality}\label{sec-unified}
We now establish the extremality of both the hyperbolic spiral map $S_k$ and the non-symplectic stretch map $s_k$ using a single, unified argument based on the modulus of curve families.

\begin{proof}[Alternative Proof of Theorems \ref{thm-spiral} and \ref{thm-stretch}]
Let $f_0: A_{1,R} \to f_0(A_{1,R})$ represent either the spiral map $S_k$ or the stretch map $s_k$. Both mappings are quasiconformal with a strictly constant maximal dilatation $K_0:=K(f_0)$. 

Let $f$ be any competing quasiconformal map that is homotopic to $f_0$ and shares its boundary values. We select a specific family of locally rectifiable curves $\Gamma$ in $A_{1,R}$ aligned with the major axis of the Beltrami ellipse of $f_0$:
\begin{itemize}
    \item If $f_0 = s_k$, we let $\Gamma = \Gamma'$, the family of simple closed curves separating the boundary components $r=1$ and $r=R$.
    \item If $f_0 = S_k$, we let $\Gamma = \Gamma_\alpha$, the family of logarithmic spirals traversing the annulus at the constant angle $\alpha$ dictated by the affine shear.
\end{itemize}

Because $f$ matches $f_0$ on $\partial A_{1,R}$ and lies in the same homotopy class, it maps the family $\Gamma$ to the topologically equivalent family $f(\Gamma) = f_0(\Gamma)$. 

By its precise geometric definition, the extremal map $f_0$ acts as a pure stretch along the trajectories of $\Gamma$, meaning it dilates the modulus of this specific curve family by exactly its maximal dilatation $K_0$. Thus, we have:
\[
\operatorname{Mod}(f_0(\Gamma)) = K_0 \operatorname{Mod}(\Gamma).
\]
Applying the fundamental quasiconformal modulus inequality \eqref{eq:mod-qc-ineq} to the competing map $f$ over the family $\Gamma$ yields:
\[
\operatorname{Mod}(f(\Gamma)) \le K(f)\cdot \operatorname{Mod}(\Gamma), \quad K(f)=\sup_{z \in A_{1,R}} K_f(z).
\]
Since $f(\Gamma) = f_0(\Gamma)$, we equate the moduli to obtain:
\[
K_0 \operatorname{Mod}(\Gamma) \le K(f)\cdot \operatorname{Mod}(\Gamma).
\]
Dividing both sides by $\operatorname{Mod}(\Gamma) > 0$, we immediately conclude that
\[
K(f)=\sup_{z \in A_{1,R}} K_f(z) \ge K_0.
\]
Therefore, $f_0$ minimises the maximal dilatation among all competing quasiconformal maps in its homotopy class, proving extremality for both $S_k$ and $s_k$ simultaneously.
\end{proof}

\section{Conclusions}\label{sec-conclusion}
To conclude, we point out the main geometric findings derived from our calculations. We initially used standard methods and Teichm\"uller theory to prove our extremality results. However, our alternative proofs using hyperbolic moduli show that we can find the extremality of both the stretch map $s_k$ and the spiral map $S_k$ using one direct method. By relying only on the standard modulus inequality, we avoid the complex analysis required by Teichm\"uller's extremality theorem.

The proofs in Section \ref{sec-unified} show how the Beltrami coefficient controls the geometry. Whether the map stretches or shears, we only need to align the curve family with the mapping's maximum distortion direction to find the global lower bound. This gives us a flexible tool to study quasiconformal maps in other spaces, like hyperbolic surfaces and the Heisenberg group.

Furthermore, returning to the motivation provided by the Collar Theorem, our exact extremality result for the hyperbolic spiral map $S_k$ provides a precise geometric limit for twisting deformations. Since a neighbourhood of a simple closed geodesic on a hyperbolic Riemann surface can be modelled as a hyperbolic circular annulus, Theorem 3.1 quantifies the absolute minimum quasiconformal distortion required to introduce a Fenchel-Nielsen twist. Finally, because these planar maps lift to solvable Lie groups, our unified proofs directly support the study of contact quasiconformal mappings in sub-Riemannian settings, translating 2D area preservation into 3D volume constraints \cite{BaloghBubaniPlatis25a, BaloghBubaniPlatis25b, BaloghFasslerPlatis}.

Finally, comparing the symplectic, almost symplectic, and non-symplectic stretch maps highlights the counterbalance between minimising local distortion and preserving overall volume. These 2D models provide a strong foundation for studying similar volume-preserving maps in higher-dimensional spaces.

\medskip

{\it Declaration.} The author acknowledges support from the Medicus programme (grant no.\ 83765).



\end{document}